\documentclass[11pt,thmsa]{article}
\usepackage{amsmath, latexsym, amsfonts, amssymb, amsthm, amscd}
\usepackage{color}

\newtheorem{theorem}{Theorem}
\newtheorem{corollary}{Corollary}

\newtheorem{lemma}{Lemma}

\newcommand{\p}{\Bbb{P}}

\newcommand{\e}{\Bbb{E}}

\newcommand{\R}{\mathbb{R}}

\newcommand{\ud}{\mathrm{d}}

\title{
\textbf{Occupation times of refracted
L\'evy processes.}}

\author{\textbf{ A.E. Kyprianou\footnote{$^{,\S}$Department of Mathematical Sciences, University of Bath, Claverton Down, {\sc Bath, BA2 7AY, United Kingdom}. Email: a.kyprianou@bath.ac.uk, \, jcpm20@bath.ac.uk} \footnote{Corresponding author.}\,, J.C. Pardo\footnote{Centro de Investigaci\'on en Matem\'aticas A.C. Calle Jalisco s/n. C.P. 36240, {\sc Guanajuato, Mexico.}  Email: jcpardo@cimat.mx}\,, J. L. P\'erez\footnote{Department of Statistics, ITAM. Rio Hondo No. 1, Col. Progreso Tizap\'an, C.P. 01080 {\sc Mexico, D.F., Mexico.} Email: jose.perez@itam.mx}, }}
\date{\footnotesize This version: \today}

\begin{document}

\maketitle

\begin{abstract}
\bigskip
A refracted L\'evy process is a L\'evy process whose dynamics change by subtracting off a fixed linear drift (of suitable size) whenever the aggregate process is above a pre-specified level. More precisely, whenever it exists, a refracted L\'evy process is described by the unique strong solution to the stochastic differential equation
\[
\ud U_t=-\delta\mathbf{1}_{\{U_t>b\}}\ud t +\ud X_t,
\]
where $X=(X_t, t\ge 0)$ is a L\'evy process with law $\p$ and $b,\delta\in \R$ such that the resulting process $U$ may visit the half line $(b,\infty)$ with positive probability. In this paper, we consider the case that $X$ is spectrally negative and establish a number of identities for the following functionals
\[
\int_0^\infty\mathbf{1}_{\{U_t<b\}}\ud t, \quad\int_0^{\rho_a^+}\mathbf{1}_{\{U_t<b\}}\ud t, \quad\int_0^{\rho^-_c}\mathbf{1}_{\{U_t<b\}}\ud t, \quad\int_0^{\rho_a^+\land\rho^-_c}\mathbf{1}_{\{U_t<b\}}\ud t,
\]
where $\rho^+_a=\inf\{t\ge 0: U_t> a\}$  and $\rho^-_c=\inf\{t\ge 0: U_t< c\}$ for $c<b<a$. Our identities  extend recent results of Landriault et al. \cite{LRZ} and bear relevance to Parisian-type financial instruments and insurance scenarios.

\bigskip

\end{abstract}

\vspace{0.5cm}
\section{Introduction and main results.}

Let $X=(X_t, t\geq 0)$ be a L\'evy process defined on a  probability space $(\Omega, \mathcal{F}, \p)$.  For $x\in \R$ denote by $\p_x$ the law of $X$ when it is started at $x$ and write for convenience  $\p$ in place of $\p_0$. Accordingly, we shall write $\e_x$ and $\e$ for the associated expectation operators. In this paper we shall assume throughout that $X$ is \textit{spectrally negative} meaning here that it has no positive jumps and that it is not a subordinator. It is well known that the latter allows us to talk about the Laplace exponent $\psi(\theta):[0,\infty) \to \R$, i.e.
\[
\e\Big[{\rm e}^{\theta X_t}\Big]=:{\rm e}^{t\psi(\theta)}, \qquad t, \theta\ge 0,
\]
and the Laplace exponent is given by the L\'evy-Khintchine formula
\begin{equation}\label{lk}
\psi(\theta)=\gamma\theta+\frac{\sigma^2}{2}\theta^2+\int_{(-\infty,0)}\big({\rm e}^{\theta x}-1-\theta x\mathbf{1}_{\{x>-1\}}\big)\Pi(\ud x),
\end{equation}
where $\gamma\in \R$, $\sigma^2\ge 0$ and $\Pi$ is a measure on $(-\infty,0)$ called the L\'evy measure of $X$ and satisfies
\[
\int_{(-\infty,0)}(1\land x^2)\Pi(\ud x)<\infty.
\]
The reader is referred to Bertoin \cite{B} and Kyprianou \cite{K} for a complete introduction to the theory of L\'evy processes.

It is well-known that $X$ has paths of bounded variation if and only if $\sigma^2=0$ and $\int_{(-1, 0)} x\Pi(\mathrm{d}x)$ is finite. In this case $X$ can be written as
\begin{equation}
X_t=ct-S_t, \,\,\qquad t\geq 0,
\label{BVSNLP}
\end{equation}
where $c=\gamma-\int_{(-1,0)} x\Pi(\mathrm{d}x)$ and $(S_t,t\geq0)$ is a driftless subordinator. Note that  necessarily $c>0$, since we have ruled out the case that $X$ has monotone paths. In this case its Laplace exponent is given by
\begin{equation*}
\psi(\lambda)= \log \mathbb{E} \left[ \mathrm{e}^{\lambda X_1} \right] = c \lambda-\int_{(-\infty,0)}\big(1- {\rm e}^{\lambda x}\big)\Pi(\ud x)
\end{equation*}

In this paper, we study occupation times of a spectrally negative L\'evy processes when its path are perturbed in a simple way. Informally speaking, a linear drift at rate $\delta>0$ is subtracted from the increments of $X$ whenever it exceeds a pre-specified positive level $b>0$. More formally, we are interested in the process $U$ which is a solution to the stochastic differential equation given by
\begin{equation}\label{SDE}
\ud U_t=\ud X_t-\delta\mathbf{1}_{\{U_t>b\}}\ud t, \qquad t\geq 0.
\end{equation}
In order to work with the above process we make the following
assumption
\begin{equation}
\mathrm{({\bf H})}\qquad\delta<\gamma-\int_{(-1,0)} x\Pi(\mathrm{d}x),\qquad\text{if $X$ has paths of bounded variation.}\notag
\end{equation}
According to Kyprianou and Loeffen \cite{KL}, this ensures that a strong solution to (\ref{SDE}) exists and the path of $U$ is not monotone.

 The special case of $X$ given in  (\ref{BVSNLP}) with compound Poisson jumps described above may also be seen as an example of a Cram\'er-Lundberg process as soon as $\mathbb{E}(X_1)>0$. This provides a specific motivation for the study of the dynamics of (\ref{SDE}). Indeed very recent studies of problems related to ruin in insurance risk has seen some preference to working with general spectrally negative L\'evy processes in place of the classical Cram\'er-Lundberg process (which is itself an example of the former class). See for example \cite{APP2007,F1998, HPSV2004a,  HPSV2004b, KKM2004, KK2006, KP2007, RZ2007, SV2007}. Under such a general model, the solution to the stochastic differential equation (\ref{SDE}) may now be thought of as the aggregate of the insurance risk process when dividends are paid out at a rate $\delta$ whenever it exceeds the level $b$.

In this paper, we consider a number of occupation  identities for the refracted process $U$, namely the following functionals
\begin{equation}
\int_0^\infty\mathbf{1}_{\{U_t<b\}}\ud t, \quad\int_0^{\rho_a^+}\mathbf{1}_{\{U_t<b\}}\ud t, \quad\int_0^{\rho^-_c}\mathbf{1}_{\{U_t<b\}}\ud t, \quad\int_0^{\rho_a^+\land\rho^-_c}\mathbf{1}_{\{U_t<b\}}\ud t,
\label{fourfunctionals}
\end{equation}
where 
\[
\rho^+_a=\inf\{t\ge 0: U_t> a\} \text{ and }\rho^-_c=\inf\{t\ge 0: U_t< c\},
\]
 for $c<b<a$. Our identities  extend recent results of Landriault et al. \cite{LRZ} where it is explained how such functionals bear relevance to, so-called, insurance risk models with Parisian implementing delays.  Indeed suppose that dividends are paid at rate $\delta$ from a surplus process $X$, modelled as a spectrally negative L\'evy process, whenever the aggregate is positive valued. In that case, the refracted L\'evy process, U, given by (\ref{SDE}) with $b=0$ plays the role of the aggregate surplus process. A Parisian-style ruin problem would declare the insurance company ruined if it remained with a negative surplus for too long. To be specific, at each time the refracted surplus process goes negative, an independent exponential clock with rate $q$ is started. If the clock rings before the refracted surplus becomes positive again then the insurance company is ruined. Assuming that the refracted process drifts to $+\infty$, and the initial value of the surplus is $x>0$, the probability of ruin can now be identified as 
 \[
 1- \mathbb{E}_x\left(\exp\left\{  - \int_0^\infty \mathbf{1}_{\{U_t<0\}}{\rm d}t\right\}\right).
 \] 
 See  \cite{LRZ-0} for further discussion.

A key element of the forthcoming analysis relies on the theory of so-called scale functions for spectrally negative L\'evy processes. We therefore devote some time in this section reminding the reader of some fundamental properties of scale functions as well as their relevance to refraction strategies.

For
each $q\geq0$ define  $W^{(q)}:
\R\to [0, \infty),$ such that $W^{(q)}(x)=0$ for all $x<0$ and on $(0,\infty)$ is the unique continuous function with Laplace transform
\begin{eqnarray}
\int^{\infty}_0\mathrm{e}^{-\theta x}W^{(q)}(x)dx=\frac1{\psi(\theta)-q},
\qquad \theta>\Phi(q),
\label{scaleLT}
\end{eqnarray}
where $ \Phi(q) = \sup\{\lambda \geq 0: \psi(\lambda) = q\}$ which is well defined and finite for all $q\geq 0$, since $\psi$ is a strictly convex function satisfying $\psi(0) = 0$ and $\psi(\infty) = \infty$. For convinience, we write $W$ instead of $W^{(0)}$. Associated to the functions $W^{(q)}$ are the functions $Z^{(q)}:\R\to[1,\infty)$ defined by
\[
Z^{(q)}(x)=1+q\int_{0}^x W^{(q)} (y)\ud y,\qquad q\ge 0.
\]
Together, the functions $W^{(q)}$ and $Z^{(q)}$ are collectively known as $q$-scale functions and predominantly appear in almost all fluctuations identities for spectrally negative L\'evy processes.  

When $X$ has paths of bounded variation, without further assumptions, it can only be said that the function $W^{(q)}$ is almost everywhere differentiable on $(0,\infty)$. However, in the case that $X$ has paths of unbounded variation, $W^{(q)}$ continuously differentiable on $(0,\infty)$; cf. Chapter 8 in \cite{K}. Throughout this text we shall write $W^{(q)\prime}$ to mean the well defined derivative in the case of unbounded variation paths and a version the density of $W^{(q)}$ with respect to Lebesgue measure in the case of bounded variation paths. This should cause no confusion as, in the latter case, $W^{(q)\prime}$ will accordingly only appear inside  Lebesgue integrals.

\bigskip

%In the following $U=(U_t,t\geq0)$ is the solution to (\ref{SDE}) when driven by $X$ and the level $b>0$. We shall frequently refer to the stopping times ...\\

We complete this section by stating our main results for the occupations measures mentioned in (\ref{fourfunctionals}). Let us note that all the identities we present essentially follow from the first identity in Theorem \ref{main} for occupation up to the stopping time $\rho_a^+\land\rho^-_c$, where $a<b<c$, by taking limits as $a\downarrow -\infty$ and $c\uparrow\infty$. We present all our results under the measure $\p_b$, that is to say, when $U$ is issued from the barrier $b$. %However, in the Appendix we show how all these results, and indeed several more, can be extracted in a straightforward way from those presented below to establish the same identities, but under the measure $\p_x$ for any $x\in\mathbb{R}$.

In what follows we recall that, for each $q\geq 0$, $W^{(q)}$ is the $q$-scale function associated to $X$, however, we shall also write $\mathbb{W}^{(q)}$ for the $q$-scale function associated to the spectrally negative L\'evy process with Laplace exponent $\psi(\theta) - \delta\theta$, $\theta\geq 0$.  We shall also write $\varphi$ for the right inverse of this Laplace exponent;
 that is to say
\begin{equation}
\varphi(q)=\sup\{\theta>0: \psi(\theta)-\delta\theta=q\},\notag
\end{equation}
for $q\geq 0$.

\begin{theorem}\label{main} Fix $\theta\geq 0$.
For $a<b<c$ we have 
\begin{equation}\label{ui}
\begin{split}
\mathbb{E}_b\bigg[\exp&\left\{-\theta\int_0^{\rho_{c}^+\wedge\rho_{a}^-}\mathbf{1}_{\{U_s<b\}}{\rm d}s\right\}\bigg]\\
&\hspace{1cm}=\frac{1/\mathbb{W}(c-b)+\displaystyle\frac{\sigma^2}{2}\mathcal{C}^{(\theta)}(a,b)+\displaystyle\int_0^{\infty}\int_{(-\infty,0)}{\mathcal A}^{(\theta)}(z,a,b,c,y)\Pi(dz-y){\rm d}y}{(\psi'(0+)-\delta)^++\displaystyle\frac{\sigma^2}{2}\mathcal{D}^{(\theta)}(a,b,c)+\displaystyle\int_0^{\infty}\int_{(-\infty,0)}\mathcal{B}^{(\theta)}(z,a,b,c,y)\Pi(dz-y){\rm d}y}.
\end{split}
\end{equation}
Where
\begin{align}
{\mathcal A}^{(\theta)}(z,a,b,c,y)&=
\frac{\mathbb{W}(c-b-y)}{\mathbb{W}(c-b)}\left(Z^{(\theta)}(z+b-a)-Z^{(\theta)}(b-a)\frac{W^{(\theta)}(z+b-a)}{W^{(\theta)}(b-a)}\right)\mathbf{1}_{(a-b,0)}(z)\mathbf{1}_{(0,c-b)}(y),\notag\\
\mathcal{B}^{(\theta)}(z, a,b, c,y)&=
{\rm e}^{-\varphi(0)y}-\frac{\mathbb{W}(c-b-y)}{\mathbb{W}(c-b)}\frac{W^{(\theta)}(z+b-a)}{W^{(\theta)}(b-a)}\mathbf{1}_{(a-b,0)}(z)\mathbf{1}_{(0,c-b)}(y),\notag\\
\mathcal{C}^{(\theta)}(a,b)&=Z^{(\theta)}(b-a)\frac{W^{(\theta)\prime}(b-a)}{W^{(\theta)}(b-a)} - \theta
W^{(\theta)}(b-a),\notag\\
\mathcal{D}^{(\theta)}(a,b,c)&=\frac{\mathbb{W}'(c-b)}{\mathbb{W}(c-b)}+\frac{W^{(\theta)\prime}(b-a)}{W^{(\theta)}(b-a)}-\varphi(0).\notag
\end{align}
\end{theorem}

\begin{corollary} \label{corr1}Fix $\theta\geq 0$. 
\begin{itemize}
\item[(i)] For $c>b$, 
\[
\begin{split}
\mathbb{E}_b&\bigg[\exp\bigg\{-\theta\int_0^{\rho_{c}^+}\mathbf{1}_{\{U_s<b\}}\ud s\bigg\}\bigg]\\
&\hspace{.4cm}= \frac{1/\mathbb{W}(c-b)}{(\psi'(0+)-\delta)^++\displaystyle\frac{\sigma^2}{2}\left(\frac{\mathbb{W}'(c-b)}{\mathbb{W}(c-b)}+ \Phi(\theta) - \varphi(0)\right)+\displaystyle\int_0^{\infty}\int_{(-\infty,0)}\mathcal{B}^{(\theta)}(z, b, c,y)\Pi(\ud z-y)\ud y},
\end{split}
\]
where
\[
\mathcal{B}^{(\theta)}(z, b, c,y)=
{\rm e}^{-\varphi(0)y}-\frac{\mathbb{W}(c-b-y)}{\mathbb{W}(c-b)}{\rm e}^{\Phi(\theta)z}\mathbf{1}_{(0,c-b)}(y).
\]

\item[(ii)] For $a<b$,
\[
\begin{split}
\mathbb{E}_b&\bigg[\exp\bigg\{-\theta\int_0^{\rho_a^-}\mathbf{1}_{\{U_s<b\}}\ud s\bigg\}\bigg]\\
&\hspace{1.4cm}=\frac{(\psi'(0+)-\delta)^+ +\displaystyle \frac{\sigma^2}{2}\mathcal{C}^{(\theta)}(a,b)+\displaystyle\int_0^{c-b}\int_{(a-b,0)}A^{(\theta)}(z,a,b){\rm e}^{-\varphi(0)y}\Pi(\ud z-y)\ud y}{(\psi'(0+)-\delta)^+ + \displaystyle\frac{\sigma^2}{2}\frac{W^{(\theta)\prime}(b-a)}{W^{(\theta)}(b-a)}+\displaystyle\int_0^{\infty}\int_{(-\infty,0)}\mathcal{B}^{(\theta)}(z, a,b,y)\Pi(\ud z-y)\ud y},
\end{split}
\]
where 
\begin{align}
\mathcal{A}^{(\theta)}(z,a,b)&=Z^{(\theta)}(z+b-a)-Z^{(\theta)}(b-a)\frac{W^{(\theta)}(z+b-a)}{W^{(\theta)}(b-a)},\label{A-theta}\\
\mathcal{B}^{(\theta)}(z, a,b,y)&={\rm e}^{-\varphi(0)y}\left(1-\frac{W^{(\theta)}(z+b-a)}{W^{(\theta)}(b-a)}\mathbf{1}_{(a-b,0)}(z)\right).\notag
\end{align}

\end{itemize}
\end{corollary}

\begin{corollary}\label{corr2}
Fix $\theta\geq 0$ and assume that $\psi'(0+)>\delta$. Then 
\begin{equation}
\mathbb{E}_b\bigg[\exp\bigg\{-\theta\int_0^{\infty}\mathbf{1}_{\{U_s<b\}}\ud s\bigg\}\bigg]
= \frac{(\psi'(0+) - \delta)\Phi(\theta)}{\theta - \delta\Phi(\theta)}.
\label{totaloccupation}
\end{equation}
Moreover,  the occupation time of $U$ below level $b$ has a density which satisfies
\[
\p_b\left(\int_0^{\infty}\mathbf{1}_{\{U_s<b\}}\ud s \in \ud x\right)= 
\left(\frac{\psi'(0+)}{\delta} - 1\right)\left(\frac{\delta\emph{\texttt{a}}}{1- \delta\emph{\texttt{a}}}\delta_0({\rm d}x) + \mathbf{1}_{\{x>0\}}\sum_{n\geq 1}\delta^n\nu^{*n}({\rm d}x)\right),
\] 
where $\delta_0({\rm d}x)$ is the Dirac-delta measure assigning unit mass to the point zero and necessarily, for $\theta\geq 0$,
\[
\frac{\Phi(\theta)}{\theta} = \emph{\texttt{a}} + \int_0^\infty {\rm e}^{-\theta x}\nu({\rm d}x).
\]
\end{corollary}

Note that from this last corollary, we easily recover Theorem 1 of Landriault et al. \cite{LRZ}. Specifically, when there is no refraction and $\delta = 0$, providing $X$ drifts to $+\infty$, that is to say $\psi'(0+)>0$, we have the occupation of $X$ below level $b$
\[
\mathbb{E}_b\bigg[\exp\bigg\{-\theta\int_0^{\infty}\mathbf{1}_{\{X_s<b\}}\ud s\bigg\}\bigg]
= \frac{\psi'(0+) \Phi(\theta)}{\theta },
\]
and its density is given by
\[
\p_b\left(\int_0^{\infty}\mathbf{1}_{\{X_s<b\}}\ud s \in \ud x\right)=\psi'(0+) (\texttt{a}\delta_0({\rm d}x)+\nu(\ud x)),
\]
 which is nothing more than the Sparre-Andersen identity.
Similarly one easily checks that by taking $\delta\downarrow 0$  in the identity given in part (ii) of Corrollary \ref{corr1}, one recovers the statement of Theorem 2 in \cite{LRZ}.

The method we shall use to prove the above results is somewhat different to the techniques employed by \cite{LRZ} and we appeal directly to the simple idea of Bernoulli trials that lies behind the excursion theory of strong Markov process, such as $U$ is. As we have no information about the excursion measure of $U$ from $b$, we initially perform the analysis for the case that $X$ has paths of bounded variation. In that case the process $U$ will almost surely take a strictly positive amount of time before it jumps below $b$ and we can construct its excursions from piecewise trajectories of spectrally negative L\'evy processes. Kyprianou and Loeffen \cite{KL} showed that in the case that $X$ has unbounded variation, refracted L\'evy processes may be constructed as the almost sure uniform limit of a sequence of bounded variation refracted L\'evy processes.  Taking account of the fact that, thanks to the continuity theorem for Laplace transforms,  scale functions are continuous in the Laplace exponent of the underlying L\'evy process, which itself is continuous in the L\'evy triplet $(\gamma, \sigma, \Pi)$ (where we understand continuity in the measure $\Pi$ to mean in the sense of weak convergence), we use an approximation procedure to derive our results for the case of unbounded variation paths from the case of bounded variation paths. 

Note that all our results can be established  at any starting point $x$ at the cost of more complicated expressions. Such identities follow from the Markov property  at the hitting times $\rho^+_b$ or $\rho^-_b$, according to whether starting point $x$ is  smaller or bigger  than $b$, and our results, we leave the details to the reader.   

The remainder of the paper is structured as follows. In the next chapter we give the proof of Theorem \ref{main} for the case that $X$ has paths of bounded variation. Thereafter, in Section \ref{sigma=0}, by means of an approximation with refracted processes having bounded variation paths, we derive the identity for the case that $X$ has paths of unbounded variation, but no Gaussian component. Finally, in Section \ref{sigma>0} we derive the missing case that $X$ has paths of unbounded variation with a Gaussian component, again by an approximation scheme, this time using a sequence of refracted L\'evy processes with unbounded variation paths and no Gaussian component. In the final section, we give some remarks about how the Corollaries \ref{corr1} and \ref{corr2} can be derived from Theorem \ref{main} by taking appropriate limits.

%%%%%%%%%%%%%%%%%%%%%%%%%%%%%%%%%%%%%%%%%%%%%%%%%%%%%%%%%%%%%%%%%%%%%%%%%%%%%%%%%%%%%%%%%%%%%%%%%%%%%%%%%%%%%%%%%%%%%%%%%%%%%%%%%%%%%%%%%%%%%%%%
\section{Proof of Theorem \ref{main}: bounded variation paths}

Let $Y = (Y_t, t\geq 0)$, where $Y_t:=X_t-\delta t,$ for $t \geq 0$ and recall that for each $q\geq 0$,  $W^{(q)}$ and $\mathbb{W}^{(q)}$ denote   the scale functions of the L\'evy processes $X$ and $Y$ respectively with $W: = W^{(0)}$ and $\mathbb{W} =\mathbb{W}^{(0)}$. Moreover, $\varphi$ is defined as the right-inverse of the Laplace exponent of $Y$.
We define the following first passages times for  $X$ and $Y$,
\begin{equation}
\tau_b^-=\inf\{t>0:X_t<b\},\qquad \tau_a^+=\inf\{t>0:X_t>a\},
\end{equation}
\begin{equation}
\kappa_b^-=\inf\{t>0:Y_t<b\},\qquad \kappa_a^+=\inf\{t>0:Y_t>a\}.
\end{equation}

Let $a<b<c$ and recall that
\begin{align}
\rho_{a}^-=\inf\{t>0:U_t< a\}\qquad\text{ and }\qquad
\rho_{c}^+=\inf\{t>0:U_t> c\}.\notag
\end{align}
We are interested in the  quantity
\begin{equation}\label{ott}
\mathbb{E}_b\left[\exp\left\{-\theta\int_0^{\rho_a^-\wedge\rho_{c}^+}\mathbf{1}_{\{U_s<b\}}\ud s\right\}\right].
\end{equation}
A crucial point in our analysis is that $b$ is irregular for $(-\infty,b)$ for $Y$ on account of it having paths of bounded variation and moreover that $Y$ does not creep downwards. This means that each excursion of $U$ from $b$ consists of a copy of $(Y_t, t\leq \kappa^-_b)$ issued from $b$ and, on the event $\{\kappa^-_b<\infty\}$, the excursion continues from the time $\kappa^-_b$ as an independent copy of $(X_t, t\leq \tau^+_b)$ issued from the randomised initial position $Y_{\kappa^-_b}$.

Here, we  express (\ref{ott}) in terms of the excursions of the process $U$ confined in the interval $[a,c]$ and, subsequently, the first excursion that exists  $[a,c]$.   Let  $(\xi_s^{(i)}, 0\le s\le \ell_i)$ be the  $i$-th excursion  of $U$ away from $b$ that does not exit  $[a,c]$, here $\ell_i$ denotes the length of the excursion at the moment it exits the interval $[a,c]$. Similarly, let $(\xi^*_s, 0\le s\le \ell_*)$  be the first excursion  of $U$ away from $b$ that exits the interval $[a,c]$ and $\ell_*$  its length.
From the strong Markov property, it is clear that the random variables $\int_0^{\ell_i}\mathbf{1}_{\{\xi^{(i)}_s<b\}}\ud s$  are i.i.d. and  independent of $\int_0^{\ell_*}\mathbf{1}_{\{\xi^*_s<b\}}\ud s$.  Set $\zeta = \inf\{t>0 : U_t = b\}$, let $E$ be the event $\{\sup_{t\leq \zeta}U_t \leq c, \, \inf_{t\leq \zeta} U_t \geq a\}$ and
$p=\p_b(E)$. A standard description of excursions of $U$ away from $b$, but confined to the interval $[a,c]$, dictates that the number of finite excursions is distributed according to an independent geometric  random variable, say $G_p$, (supported on $\{0,1,2,\ldots\}$) with parameter $p$, the random variables $\int_0^{\ell_i}\mathbf{1}_{\{\xi^{(i)}_s<b\}}\ud s$ are equal in distribution to $\int_0^\zeta \mathbf{1}_{\{U_s<b\}}{\rm d}s$ under the conditional law $\p_b(\cdot|E)$ and the random variable $\int_0^{\ell_*}\mathbf{1}_{\{\xi^*_s<b\}}\ud s$ is equal in distribution $\int_0^{\rho_a^-\wedge\rho_{c}^+} \mathbf{1}_{\{U_s<b\}}{\rm d}s$ but now under the conditional law $\p_b(\cdot|E^c)$.

It now follows that  
\begin{eqnarray}
\lefteqn{\mathbb{E}_b\bigg[\exp\bigg\{-\theta\int_0^{\rho_a^-\wedge\rho_{c}^+}\mathbf{1}_{\{U_s<b\}}\ud s\bigg\}\bigg]}&&\notag\\
&=&\e_b\left[\prod_{i=0}^{G_p}\exp\left\{-\theta\int_0^{\ell_i}\mathbf{1}_{\{\xi^{(i)}_s<b\}}\ud s\right\}\exp\left\{-\theta\int_0^{\ell_*}\mathbf{1}_{\{\xi^*_s<b\}}\ud s\right\}\right]\notag\\
&=&\mathbb{E}\left[\e_b\left[\exp\left\{-\theta\int_0^{\ell_i}\mathbf{1}_{\{\xi^{(i)}_s<b\}}\ud s\right\}\right]^{G_p}\right]\e_b\left[\exp\left\{-\theta\int_0^{\ell_*}\mathbf{1}_{\{\xi^*_s<b\}}\ud s\right\}\right],
\label{putin}
\end{eqnarray}
Recall that the generating function of the independent geometric random variable $G_p$ satisfies,
\[
F(s)=\frac{q}{1-sp},\qquad |s|<\frac{1}{p},
\]
where $q=1-p$. 
Therefore the first term of the right-hand side of the above identity satisfies
\begin{equation}
\begin{split}
\mathbb{E}\left[\e_b\left[\exp\left\{-\theta\int_0^{\ell_i}\mathbf{1}_{\{\xi^{(i)}_s<b\}}\ud s\right\}\right]^{G_p}\right]
%&=F\left(\e_b\left[\exp\left\{-\theta\int_0^{\ell_i}\mathbf{1}_{\{\xi^{(i)}_s<b\}}\ud s\right\}\right]\right)\\
&=\frac{q}{1-p\e_b\left[\exp\left\{-\theta\displaystyle\int_0^{\ell_i}\mathbf{1}_{\{\xi^{(i)}_s<b\}}\ud s\right\}\right]}.
\end{split}
\label{pieces together}
\end{equation}
%Since  $\int_0^{\ell_i}\mathbf{1}_{\{\xi^{(i)}_s<b\}}\ud s,$ corresponds to the amount of time that the  excursion $\xi^{(i)}$ spends below $b>0$, we have for each $i\geq 1$,
Moreover, again taking account of the remarks in the previous paragraph, we also have that 
\[
\begin{split}
\e_b\left[\exp\left\{-\theta\int_0^{\ell_i}\mathbf{1}_{\{\xi^{(i)}_s<b\}}\ud s\right\}\right]&=\frac{1}{p}\int_{(a,b)}\mathbb{P}_b\Big(Y_{\kappa_b^-}\in
dz,\overline{Y}_{\kappa_b^-}<c\Big)\mathbb{E}_z\Big[{\rm e}^{-\theta\tau_b^+},\tau_b^+<\tau_a^-\Big]\\
&=\frac{1}{p}\int_{(a-b,0)}\mathbb{P}\Big(Y_{\kappa_0^-}\in
dz,\overline{Y}_{\kappa_0^-}<c-b\Big)\mathbb{E}_{z+b-a}\Big[{\rm e}^{-\theta\tau_{b-a}^+},\tau_{b-a}^+<\tau_0^-\Big],
\end{split}
\]
where $\overline{Y}_t=\sup_{0\le s\le t}Y_s$. From the Compensation Formula (see for instance identity (8.27) in \cite{K}), one can deduce
\begin{equation}\label{ina}
\mathbb{P}\Big(Y_{\kappa_0^-}\in
dz,\overline{Y}_{\kappa_0^-}<c-b\Big)=\frac{\mathbb{W}(0+)}{\mathbb{W}(c-b)}\int_0^{c-b}\mathbb{W}(c-b-y)\Pi(\ud z-y)\ud y,
\end{equation}
and from identity (8.8) in \cite{K}, we have
\begin{equation}\label{inc}
\mathbb{E}_{z+b-a}\Big[{\rm e}^{-\theta\tau_{b-a}^+},\tau_{b-a}^+<\tau_0^-\Big]=\frac{W^{(\theta)}(z+b-a)}{W^{(\theta)}(b-a)}.
\end{equation}
Therefore from  (\ref{ina}) and  (\ref{inc}), we get
\[
\e_b\left[\exp\left\{-\theta\int_0^{\ell_i}\mathbf{1}_{\{\xi^{(i)}_s<b\}}\ud s\right\}\right]=\frac{\mathbb{W}(0+)}{p}\int_0^{c-b}\int_{(a-b,0)}\frac{\mathbb{W}(c-b-y)}{\mathbb{W}(c-b)}\frac{W^{(\theta)}(z+b-a)}{W^{(\theta)}(b-a)}\Pi(dz-y){\rm d}y.
\]
It is worth noting at this point that $\mathbb{W}(0+) >0$ precisely because we have restricted ourselves to the case of bounded variation paths.

The classical ruin problem for $Y$ tells us that 
\[
\p\big(\kappa_0^-<\infty\big)=
\begin{cases}
1 &\text{ if $\psi^\prime(0+)-\delta\le0$,}\\
1-(\psi^\prime(0+)-\delta)W(0+)&\text{ if $\psi^\prime(0+)-\delta>0$,}
\end{cases}
\]
see for example formula  (8.7) in \cite{K}. Taking limits as $c\uparrow\infty$ the formula with (\ref{ina}), making use of Exercise  8.5 in \cite{K} which tells  us that $\lim_{c\uparrow\infty}\mathbb{W}(c-b-y)/\mathbb{W}(c-b) = \exp\{-\varphi(0)y\}$, we get 
\[
1=\mathbb{W}(0+)(\psi^\prime(0+)-\delta)^++\mathbb{W}(0+)\int_0^\infty\int_{(-\infty,0)}\Pi(\ud z-y){\rm e}^{-\varphi(0)y}\ud y.
\]
Hence putting all the pieces together in (\ref{pieces together}), we obtain
\begin{eqnarray}
\lefteqn{\mathbb{E}\bigg[\e_b\bigg[\exp\bigg\{-\theta\int_0^{\ell_i}\mathbf{1}_{\{\xi^{(i)}_s<b\}}\ud s\bigg\}\bigg]^{G_p}\bigg]}&&\notag\\
&=&\frac{q}{\mathbb{W}(0+)(\psi^\prime(0+)-\delta)^++\mathbb{W}(0+)\displaystyle\int_0^{\infty}\int_{(-\infty,0)}\mathcal{B}^{(\theta)}(z,a,b,c,y)\Pi(\ud z-y)\ud y},
\label{shitty}
\end{eqnarray}
where
\begin{equation}
\mathcal{B}^{(\theta)}(z,a,b,c,y)={\rm e}^{-\varphi(0)y}-\frac{\mathbb{W}(c-b-y)}{\mathbb{W}(c-b)}\frac{W^{(\theta)}(z+b-a)}{W^{(\theta)}(b-a)}\mathbf{1}_{(a-b,0)}(z)\mathbf{1}_{(0,c-b)}(y).\notag
\end{equation}

Next, we compute  the Laplace transform of  $\int_0^{\ell_*}\mathbf{1}_{\{\xi^*_s<b\}}\ud s$.  Recalling that $\ell_*$ is equal in law to $\rho_a^-\wedge\rho_{c}^+$ under $\p_b(\cdot|E^c)$ we have two  cases to consider. In the first case,   the process $Y$ continuously exits the interval $[b,c]$ at $c$. In the second case, the process $U$ exits  the interval $[b,c]$ downwards by a jump, if the process jumps into $[a,b)$,  it continues until it jumps again below $a$. Hence  from identities (8.8), (8.9) in \cite{K} and (\ref{ina}), we have
\begin{equation}\label{bigshit}
\begin{split}
\e_b\bigg[&\exp\bigg\{-\theta\int_0^{\ell_*}\mathbf{1}_{\{\xi^*_s<b\}}\ud s\bigg\}\bigg]\\\
&=\frac{1}{q}\bigg(\mathbb{P}_b\Big(\kappa_c^+<\kappa_b^-\Big)+\int_{(a,b)}\mathbb{P}_b\Big(Y_{\kappa_b^-}\in
\ud z,\overline{Y}_{\kappa_b^-}<c\Big)\mathbb{E}_z\Big[{\rm e}^{-\theta\tau_a^-},\tau_a^-<\tau_b^+\Big]\bigg)\\
&=\frac{1}{q}\bigg(\frac{\mathbb{W}(0+)}{\mathbb{W}(c-b)}+\int_{(a-b,0)}\mathbb{P}\Big(Y_{\kappa_0^-}\in
\ud z,\overline{Y}_{\kappa_0^-}<c-b\Big)\mathbb{E}_{z+b-a}\Big[{\rm e}^{-\theta\tau_0^-},\tau_0^-<\tau_{(b-a)}^+\Big]\bigg)\\
&=\frac{1}{q}\bigg(\frac{\mathbb{W}(0+)}{\mathbb{W}(c-b)}+\mathbb{W}(0+)\int_0^{c-b}\int_{(a-b,0)}\bigg(Z^{(\theta)}(z+b-a)-Z^{(\theta)}(b-a)\frac{W^{(\theta)}(z+b-a)}{W^{(\theta)}(b-a)}\bigg)\\
&\hspace{10cm}\times\frac{\mathbb{W}(c-b-y)}{\mathbb{W}(c-b)}\Pi(\ud z-y)\ud y
\bigg).
\end{split}
\end{equation}
Let
\[
{\mathcal A}^{(\theta)}(z,a,b,c,y)=A^{(\theta)}(z,a,b)\frac{\mathbb{W}(c-b-y)}{\mathbb{W}(c-b)}\mathbf{1}_{(a-b,0)}(z)\mathbf{1}_{(0,c-b)}(y),
\]
where we recall $A^{(\theta)}(z, a, b)$ was defined in (\ref{A-theta}).

Plugging (\ref{shitty}) and (\ref{bigshit}) back into (\ref{putin}) we get the desired identity.  \hfill$\square$
%\begin{equation}\label{fi2}
%\mathbb{E}_b\bigg[\exp\bigg\{-\theta\int_0^{\rho_a^-\wedge\rho_{c}^+}\mathbf{1}_{\{U_s<b\}}\ud s\bigg\}\bigg]=\frac{\displaystyle\frac{1}{\mathbb{W}(c-b)}+\displaystyle\int_{(-\infty,0)}\int_0^{\infty}{\mathcal A}^{(\theta)}(z,a,b,c,y)\Pi(\ud z-y)\ud y}{(\psi'(0+)-\delta)^++\displaystyle\int_{(-\infty,0)}\int_0^{\infty}\mathcal{B}^{(\theta)}(z,a,b,c,y)\Pi(\ud z-y)\ud y}.
%\end{equation}

\section{Proof of Theorem \ref{main}: unbounded variation paths, $\sigma^2=0$}\label{sigma=0}

In this part of the proof we extend the previous calculations to unbounded variation
L\'evy process with no Gaussian component ($\sigma =0$). 
There is a very particular reason why we do not consider the inclusion of the Gaussian component, which is related to smoothness properties of scale functions. We shall address this issue  at the end of the section.

To start with we recall the following well established result
which can be found discussed, for example on p.210 of  \cite{B}.
For any spectrally negative L\'evy process with
unbounded variation paths $X$, there exists a sequence of bounded
variation spectrally negative Levy processes $X^n = (X^n_t, t\geq 0)$, $n\geq 1$, such that for
each $t > 0$,
\begin{equation}
\lim_{n\to\infty} \sup_{s\in[0,t ]} \big|X^n_s-X_s\big| = 0, \qquad \textrm{a.s.}
\label{uniformBertoin}
\end{equation}
Moreover, when $X^n$ is written in the form (\ref{BVSNLP})
the drift coefficient tends to infinity as $n\uparrow\infty$. The
latter  implies that for all $n$ sufficiently large, the
sequence $X^n$ will automatically fulfill  condition ({\bf H}). Such a
sequence, $X^n$ will be referred to as strongly approximating for
$X$. Rather obviously we may also talk of a strongly approximating
sequence for processes of bounded variation respecting ({\bf H}).

On the other hand following Kyprianou and Loeffen \cite{KL}, we have
also have the following Lemma.
\begin{lemma}\label{Un}
Suppose that $X$ is a spectrally negative Levy process satisfying
{\bf (H)} and that $(X^n)_{n\ge 1}$ is any strongly approximating sequence. Denote by
$U^n = (U^n_t, t\geq 0)$, ${n\ge 1}$, the sequence of pathwise solutions associated with each $X^n$. Then
there exists a stochastic process $U = (U_ t: t\geq 0)$ such that
for each fixed $t > 0$,
\begin{equation}
\lim_{n\to\infty} \sup_{s\in[0,t ]} \big|U^n_s-U_s\big| = 0,\qquad \textrm{ a.s.}
\end{equation}
\end{lemma}
Let  $X$ be a L\'evy process of unbounded variation with no Gaussian component and $U$ the
corresponding solution to (\ref{SDE}). 
%Hence,  there exists a sequence of bounded variation L\'evy processes $(X^n)_{n\ge 1}$, such that for each $n\ge 1$,  $X^n$ satisfies the conditions of the previous lemmas.  
Let us denote by $\Pi_n, W_n^{(\theta)}, Z^{(\theta)}_n$ and $\psi_n$ for the corresponding L\'evy measure, scale functions and Laplace exponent of a strongly approximating sequence $X^n$, $n\geq 1$, discussed above. For the corresponding sequences of strongly approximating refracted L\'evy processes $U^n$, $n\geq1$, given by Lemma \ref{Un},  define
\[
\rho_{x}^{+,n}=\inf\{t\ge 0: U^n_t> x\}\qquad\textrm{ and }\qquad\rho_{x}^{-,n}=\inf\{t\ge 0: U^n_t<x\},
\]
where $x\in \R$.
We also denote by $\mathbb{W}_n$, $\mathbb W^{(\theta)}_n$ and $\varphi_n$ for the scale functions and the right-continuous inverse of the Laplace exponent associated to the L\'evy process $Y^n$ which is defined by $Y^n_t=X^n_t-\delta t$, for $t\ge 0$. 

The conclusion of the previous section tells us that 
\begin{eqnarray}
\lefteqn{\hspace{-2cm}\mathbb{E}_b\left[\exp\left\{-\theta\int_0^{\rho_{c}^{+,n}\wedge\rho_{a}^{-,n}}\mathbf{1}_{\{U^n_s<b\}}\ud s\right\}\right]}&&\notag\\
&&=\frac{\displaystyle\frac{1}{\mathbb{W}_n(c-b)}+\displaystyle\int_0^{\infty}\int_{(-\infty,0)}{\mathcal A}^{(\theta),n}(z,a,b,c,y)\Pi_n(\ud z-y)\ud y}{(\psi_n'(0+)-\delta)^++\displaystyle\int_0^{\infty}\int_{(-\infty,0)}\mathcal{B}^{(\theta),n}(z,a,b,c,y)\Pi_n(\ud z-y)\ud y},
\label{fii2}
\end{eqnarray}
where ${\mathcal A}^{(\theta),n}(z,a,b,c,y)$ and 
$\mathcal{B}^{(\theta),n}(z,a,b,c,y)$ are obviously defined. Our objective is to show that both left and right hand side above converge to their respective components of (\ref{ui}).

%\[
%\mathcal{B}^{(\theta),n}(z,a,b,c,y)={\rm e}^{\varphi_n(0)y}-\frac{\mathbb{W}_n(c-b-y)}{\mathbb{W}_n(c-b)}\frac{W^{(\theta)}_n(z+b-a)}{W^{(\theta)}_n(b-a)}\mathbf{1}_{(a-b,0)}(z)\mathbf{1}_{(0,c-b)}(y)
%\]
%and
%\begin{equation}
%{\mathcal A}^{(\theta),n}(z,a,b,c,y)=\frac{\mathbb{W}_n(c-b-y)}{\mathbb{W}_n(c-b)}\left(Z^{(\theta)}_n(z+b-a)-Z^{(\theta)}_n(b-a)\frac{W^{(\theta)}_n(z+b-a)}{W^{(\theta)}_n(b-a)}\right)\mathbf{1}_{(a-b,0)}(z)\mathbf{1}_{(0,c-b)}(y)\notag.
%\end{equation}

To this end, we shall start by noting that, according to Bertoin $\cite{B}$ (see the comment on page  210), the sequence $(X^n)_{n\ge 1}$ may be constructed in such a way that its
  L\'evy measure $\Pi_n$  is given by $\Pi_n(\ud x)=\mathbf{1}_{\{x<-1/n\}}\Pi(\ud x)$ and accordingly, there is   pointwise convergence of $\psi_n(\theta)$ to $\psi(\theta)$, the Laplace exponent of the desired limiting process $X$.
This construction also ensures the convergence of $\psi'_n(0+)$ to $\psi'(0+)$. Through the Continuity Theorem for Laplace transforms, it was also shown 
 in \cite{KL} that 
\begin{equation}
\lim_{n\to\infty}W^{(\theta)}_n(x)=W^{(\theta)}(x)\qquad\text{ for all} \quad x\geq0,
\label{recall}
\end{equation}
Therefore, as $n\to\infty$, we can ensure with the aformentioned approximating sequence $(X^n)_{n\geq 1}$, we have for all $y,z$
\[
\mathcal{B}^{(\theta),n}(z,a,b,c,y)\to \mathcal{B}^{(\theta)}(z,a,b,c,y)
\quad\text{
and
}\quad
{\mathcal A}^{(\theta),n}(z,a,b,c,y)\to {\mathcal A}^{(\theta)}(z,a,b,c,y).
\]

%According to Bertoin $\cite{B}$ (see the comment on page  210) the sequence $(X^n)_{n\ge 1}$ may be constructed in such a way that its  L\'evy measure $\Pi_n$  is given by $\Pi_n(\ud x)=\mathbf{1}_{\{x<-1/n\}}\Pi(\ud x).$   Therefore, 

If we can provide appropriate uniform and integrable bounds on 
\[
|\mathcal{B}^{(\theta),n}(z,a,b,c,y)|\qquad \textrm{ and }\qquad |{\mathcal A}^{(\theta),n}(z,a,b,c,y)|,
\]
 a straightforward argument using dominated convergence will suffice to prove that the right hand side of (\ref{fii2}) converges to the desired limit.
To this end, we need the following helpful result.

\begin{lemma}\label{eta-convergence}
Let $x>0$. Suppose that $(\psi_n)_{n\geq 1}$ is a sequence of Laplace exponents of spectrally negative L\'evy processes such that $\psi_n$ tends pointwise on $(0,\infty)$ to $\psi$, which is also the Laplace exponent of a spectrally negative L\'evy process.  Denote by $(\eta_n)_{n\geq 1}$ and $\eta$
the  excursion measures of the respective processes associated to $(\psi_n)_{n\geq 1}$ and $\psi$, reflected in their supremum. Then, writing $\overline\epsilon$ for the supremum of the canonical excursion,
\begin{equation}
\eta_n(\bar{\epsilon}\geq x)\to \eta(\bar{\epsilon}\geq x)
\end{equation}
as $n\rightarrow\infty$.
\end{lemma}
\begin{proof}
Let $x>0$, then by identity Lemma 8.2 in \cite{K} we have that
\begin{equation}\label{ex1}
W_n'(x)=\eta_n(\bar{\epsilon}\geq x)W_n(x),
\end{equation}
for almost every $x>0$,
with a slightly stronger statement holding for $X$ on account of the fact that it has paths of unbounded variation.
Indeed, it is also shown  that  the function $\eta(\bar{\epsilon}\geq x)$ is continuous for all $x>0$. In that case $W'$ exists for all $x>0$ and its continuous. 

Now using Lemma 20 in \cite{KL} we
know that
\begin{equation}
W_n(x)\to W(x)\text{ as $n\to\infty$, for all } x>0,\notag
\end{equation}
and also that $W_n'\to W'$ as $n\to\infty$ for almost all
$x>0$. Then by (\ref{ex1}) we can conclude that
\begin{equation}
\eta_n(\bar{\epsilon}\geq x)\to \eta(\bar{\epsilon}\geq x)\text{ as
$n\to\infty$, for almost all $x\geq0$}.\notag
\end{equation}
Let $N=\{x>0: \lim_{n\to\infty}\eta_n(\bar{\epsilon}\geq
x)=\eta(\bar{\epsilon}\geq x)\}$, and let $z>0$ such that $z\not\in
N$. We consider two sequences $(x_n)_{n\in\mathbb{N}}\subset
N$ and $(y_n)_{n\in\mathbb{N}}\subset N$, such that
\begin{equation}
x_n\uparrow z\qquad\text{and}\qquad y_n\downarrow z\text{ as
$n\to\infty$}.\notag
\end{equation}
Then we have that
\begin{equation}
\eta_n(\bar{\epsilon}\geq y_m)\leq \eta_n(\bar{\epsilon}\geq z)\leq
\eta_n(\bar{\epsilon}\geq x_m)\, \notag
\quad\text{ and hence }\quad
\eta(\bar{\epsilon}\geq y_m)\leq
\limsup_{n\to\infty}\eta_n(\bar{\epsilon}\geq z)\leq
\eta(\bar{\epsilon}\geq x_m).\notag
\end{equation}
Finally taking $m\to\infty$ in the previous inequality we obtain
\begin{equation}
\eta(\bar{\epsilon}> z)\leq
\limsup_{n\to\infty}\eta_n(\bar{\epsilon}\geq z)\leq
\eta(\bar{\epsilon}\geq z).\notag
\end{equation}
Using the fact that $z\mapsto \eta(\bar{\epsilon}\geq z)$ is continuous, we get that
$
\limsup_{n\to\infty}\eta_n(\bar{\epsilon}\geq z)=
\eta(\bar{\epsilon}\geq z)
$
Appealing to a similar argument for   $\liminf\eta_n(\overline\epsilon \geq z)$,  we  conclude that
$
\lim_{n\to\infty}\eta_n(\bar{\epsilon}\geq z)=\eta(\bar{\epsilon}\geq
z).$
This implies that $z\in N$ which is a contradiction. The statement of the lemma now follows.
\end{proof}

  \bigskip
  
We return to the proof of Theorem \ref{main}. Note that, since  the functions $\mathbb{W}_n,W^{(\theta)}_n$ and
$Z^{(\theta)}_n$ are strictly increasing, then for $y\in(0,c-b)$ and $z\in(a-b,0)$, we have on the one hand,
\begin{eqnarray}
\Big|{\mathcal A}^{(\theta),n}(z,a,b,c,y)\Big|&=&\bigg|\frac{\mathbb{W}_n(c-b-y)}{\mathbb{W}_n(c-b)}\bigg(Z^{(\theta)}_n(z+b-a)-Z^{(\theta)}_n(b-a)\frac{W^{(\theta)}_n(z+b-a)}{W^{(\theta)}_n(b-a)}\bigg)\bigg|\notag\\
%&\leq
%Z^{(\theta)}_n(z+b-a)+Z^{(\theta)}_n(b-a)\frac{W^{(\theta)}_n(z+b-a)}{W^{(\theta)}_n(b-a)}\\
%&\leq
%Z^{(\theta)}_n(c-a)+Z^{(\theta)}_n(b-a)\frac{W^{(\theta)}_n(c-a)}{W^{(\theta)}_n(b-a)}.
%&\leq& 2Z^{(\theta)}_n(c-a).
&\leq &\left|\int^0_{-z} \bigg(\theta W^{(\theta)}_n(u+b-a) -Z^{(\theta)}_n(b-a)\frac{W^{(\theta)\prime}_n(u+b-a)}{W_n^{(\theta)}(b-a)}\bigg){\rm d}u\right|,
\label{DOM}
\end{eqnarray}
where we have used the monotonicity of $\mathbb{W}_n$ in the inequality.
%where the derivative $W^{(\theta)\prime}_n$ need only be defined almost everywhere.
Next we recall, for example, from formulae (2.18) and (2.19) of \cite{KKR} that 
\begin{equation}
\frac{W^{(\theta)\prime}_n(u)}{W_n^{(\theta)} (u)} = \eta^{\Phi_n(\theta)}_n(\overline\epsilon\geq u) + \Phi_n(\theta)
\label{W'/W}
\end{equation}
almost everywhere, where $\eta^{\Phi_n(\theta)}_n$ is the excursion measure of $X^n$ under the exponential change of measure
\[
\left.\frac{{\rm d}\mathbb{P}^{\Phi_n(\theta),n}}{{\rm d}\mathbb{P}}\right|_{\mathcal{F}_t} = {\rm e}^{\Phi_n(\theta)X^n_t - \theta t}, \qquad t\geq 0,
\] 
with $\mathcal{F}_t = \sigma(X_s: s\leq t)$.
In particular, the right hand side of (\ref{W'/W}) is a non-increasing function. It is therefore straightforward to check, with the help of Lemma \ref{eta-convergence}  (note that the Laplace exponent of $(X^n, \mathbb{P}^{\Phi_n(\theta),n})$ is equal to $\psi_n(\cdot+ \Phi_n(\theta)) - \theta$ and this tends pointwise to $\psi(\cdot + \Phi(\theta))-\theta$ on $(0,\infty)$) that the integrand on the right hand side of (\ref{DOM}) is uniformly bounded. Hence 
there exists a constant $K>0$ such that for all $y,z$ and sufficiently large $n$,
\begin{equation}
\Big|{\mathcal A}^{(\theta),n}(z,a,b,c,y)\Big|\leq K|z| \mathbf{1}_{(a-b,0)}(z)\mathbf{1}_{(0,c-b)}(y)
\label{DOM1}
\end{equation}
which is integrable with respect to $\Pi({\rm d}z  -y){\rm d}y$ on $(0,\infty)^2$.
Indeed, to confirm the latter, noting that $\Pi$ is a finite measure away from the origin, it suffices to check, with the help of Fubini's Theorem, that for sufficiently small $\varepsilon>0$,
\begin{eqnarray}
\int_0^{\varepsilon} \int_{(-\varepsilon, 0)}(-z)\Pi({\rm d}z - y){\rm d}y &=& \int_0^{\varepsilon} \int_{(-\varepsilon-y, -y)}\mathbf{1}_{(-\varepsilon, 0)}(z)(-y-z)\Pi({\rm d}z ){\rm d}y \notag\\
&=& \int_{(-\varepsilon, 0)}\int_0^\infty \mathbf{1}_{(0,\varepsilon)}(y)\mathbf{1}_{(-z-\varepsilon, -z)}(y)(-y-z){\rm d}y\Pi({\rm d}z)\notag\\
&=& -\frac{1}{2}\int_{(-\varepsilon, 0)}\left.(z+y)^2\right|_{0}^{-z}\Pi({\rm d}z)\notag\\
&=& \frac{1}{2}\int_{(-\varepsilon, 0)}z^2\Pi({\rm d}z)<\infty.\label{recycle}
\end{eqnarray}
%On the other hand, for small $\varepsilon>0$, when ......
%Noting that 
%\[
%\int_{(-\infty,0)}\int_0^{\infty} \mathbf{1}_{(a-b,0)}(z)\mathbf{1}_{(0,c-b)}(y)\Pi(\ud z-y)\ud y  = \int_0^{c-b} \Pi(-y-b+a, -y){\rm d}y=\infty
%\]
%({\bf A: problem here!})
The Dominated Convergence
Theorem now implies  that
\[
\lim_{n\to\infty}\int_{(-\infty,0)}\int_0^{\infty}{\mathcal A}^{(\theta),n}(z,a,b, c,y)\Pi_n(\ud z-y)\ud y=\int_{(-\infty,0)}\int_0^{\infty}{\mathcal A}^{(\theta)}(z,a,b,c,y)\Pi(\ud z-y)\ud y.
\]

Now let us
check the integral in the denominator of (\ref{fii2}). Proceeding as before, we note that the functions
$\mathbb{W}_n,W^{(\theta)}_n$ are strictly increasing. Then we may find an upper bound for $\mathcal{B}^{(\theta),n}(z,a,b,c,y)$ as follows

\begin{equation}\label{cc4}
\begin{split}
\Big|\mathcal{B}^{(\theta)}(z,a,b,c,y)\Big|=\bigg|{\rm e}^{-\varphi_n(0)y}&-\frac{\mathbb{W}_n(c-b-y)}{\mathbb{W}_n(c-b)}\frac{W^{(\theta)}_n(z+b-a)}{W^{(\theta)}_n(b-a)}\mathbf{1}_{(a-b,0)}(z)\mathbf{1}_{(0,c-b)}(y)\bigg|\\
&\leq
1
+\frac{\mathbb{W}_n(c-b-y)}{\mathbb{W}_n(c-b)}\frac{W^{(\theta)}_n(z+b-a)}{W^{(\theta)}_n(b-a)}\\
&\leq 2.
\end{split}
\end{equation}
It follows that, provided we consider the part of the integral 
\begin{equation}
\int_0^{\infty}\int_{(-\infty,0)}\mathcal{B}^{(\theta),n}(z,a,b,c,y)\Pi_n(\ud z-y)\ud y
\label{theintegral}
\end{equation}
which concerns values of $y$ and $z$ which are bounded away from zero, we may appeal to 
dominated convergence to pass the limit in $n$ through the integral in the obvious way.

Let us therefore turn our attention to the part of (\ref{theintegral}) which concerns small values of $y$ and $z$. First note that we can write for $\varepsilon$ sufficiently small,
\begin{eqnarray*}
\int_0^{\varepsilon}\int_{(-\varepsilon,0)}\mathcal{B}^{(\theta),n}(z,a,b,c,y)\Pi_n(\ud z-y)\ud y
&=&\int_0^{\varepsilon}\int_{(-\varepsilon-y,-y)}\mathbf{1}_{(-\varepsilon,0)}(z)\mathcal{B}^{(\theta),n}(z+y,a,b,c,y)\Pi_n(\ud z)\ud y
\end{eqnarray*}
In the spirit of (\ref{DOM}) we can write 
\[
\mathcal{B}^{(\theta),n}(z+y,a,b,c,y)- \mathcal{B}^{(\theta),n}(z,a,b,c,0)  = \int_0^y \frac{\partial}{\partial u} \mathcal{B}^{(\theta),n}(z+u,a,b,c,u){\rm d}u,
\]
where the derivative is understood as a density with respect to Lebesgue measure (on account of the fact  that scale functions only have, in general, a derivative almost everywhere). 
Using similar reasoning to the derivation of the inequality (\ref{DOM1}) we may also deduce 
that, for $0<y,|z|<\varepsilon$ and $\varepsilon$ sufficiently small,
\[
\left|\mathcal{B}^{(\theta),n}(z+y,a,b,c,y)- \mathcal{B}^{(\theta),n}(z,a,b,c,0) \right|
\leq K_1y
\]
for some constant $K_1>0$. With similar reasoning we can also check that, within the same regime of $y$ and $z$,
\[
\left|\mathcal{B}^{(\theta),n}(z,a,b,c,0) \right|\leq K_2 |z|,
\]
where $K_2>0$ is a constant. We now claim that $K_1 y + K_2|z|$ is a suitable dominating function on $0<y,|z|<\varepsilon$. Taking account of the computations in (\ref{recycle}), to verify the last claim, it suffices to check that for $\varepsilon$ sufficiently small,
\begin{eqnarray*}
\int_0^{\varepsilon}\int_{(-\varepsilon-y,-y)}\mathbf{1}_{(-\varepsilon,0)}(z)y\Pi_n(\ud z)\ud y &=&\int_{(-\varepsilon,0)} \int_0^{\infty}y
\mathbf{1}_{(0,\varepsilon)}(y)\mathbf{1}_{(-z-\varepsilon,-z)}(y)\ud y\Pi_n(\ud z)\\
&=&\int_{(-\varepsilon,0)} \int_0^{\infty}y
\mathbf{1}_{(0,-z)}(y)\ud y\Pi_n(\ud z)\\
&=&\frac{1}{2} \int_{(-\varepsilon,0)}z^2\Pi({\rm d}z)<\infty.
\end{eqnarray*}

Thus far we have shown that the right hand side of (\ref{fii2}) converges to the desired expression. To deal with the left hand side of (\ref{fii2}), let us note that, according to the proof of Lemma
VII.23 in \cite{B},  we have
that $\mathbb{P}$-a.s.
\[
\lim_{n\to\infty}\rho^{+,n}_{a}=\rho^+_a,\qquad\text{
and}\qquad\lim_{n\to\infty}\rho^{-,n}_{c}=\rho^-_c.
\]
Finally using the uniform convergence of $U^n$ to $U$ on fixed, bounded intervals of time together with the
dominated convergence Theorem, we obtain
\[
\lim_{n\to\infty}\mathbb{E}_b\left[\exp\left\{-\theta\int_0^{\rho_{a}^{+,n}\wedge\rho_{c}^{-,n}}\mathbf{1}_{\{U^n_s<b\}}\ud s\right\}\right]=\mathbb{E}_b\left[\exp\left\{-\theta\int_0^{\rho_{a}^+\wedge\rho_{c}^-}\mathbf{1}_{\{U_s<b\}}\ud s\right\}\right].
\]
Hence taking limits in both sides of (\ref{fii2}) give us the desired  result. \hfill$\square$
\bigskip

Let us conclude this section by commenting on why we have excluded the case $\sigma^2 >0$. As we have seen above, we have taken account of the continuity of scale functions with respect to the underlying L\'evy triplet in taking limits through an approximating sequence of processes. Had the target L\'evy process $X$ included a Gaussian component, then, as we shall see in the next section, we would have found the appearance of derivatives of scale functions appearing in the limit on the right hand side of (\ref{fii2}). Whilst scale functions are continuous, they do not in general have continuous derivatives. Indeed for processes of bounded variation, in the case that the L\'evy measure has atoms, the associated scale functions are at best almost everywhere differentiable. For spectrally negative L\'evy processes of unbounded variation however, scale functions are continuously differentiable suggesting that it would be more convenient to deal the case of $\sigma^2>0$ by working with an approximating sequence of unbounded variation processes. This is precisely what we do in the next section.

\section{Proof of Theorem \ref{main}: unbounded variation paths, $\sigma^2>0$}\label{sigma>0}

We now consider the case when the driving L\'evy process $X$ has a Gaussian
component, i.e. $\sigma^2>0$. We will strongly
approximate $X$ by a sequence of L\'evy process of unbounded
variation $(X^n)_{n\ge 1}$ with no Gaussian component. Again  following the discussion on p210 of Bertoin \cite{B}, we may construct  the sequence $(X^n)_{n\ge 1}$
in such a way their  L\'evy measures $\Pi_n$ satisfy
\begin{equation}
\Pi_n(\ud x)=n^2\sigma^2\delta_{-\frac{1}{n}}(\ud x)+\Pi(\ud x),
\label{approx-Pi}
\end{equation}
where $\delta_z(\ud x)$ is the Dirac measure at $z$. In order to obtain this case, we first  prove a series of useful Lemmas.
%Recall the  change of
%measure,
%\[
%\frac{\ud\mathbb{P}^{(c)}}{\ud\mathbb{P}}\bigg|_{\mathcal{F}_t}={\rm e}^{cX_t-\psi(c)t},
%\]
%also known as the Esscher transform.
\begin{lemma}\label{lema1}
Let $x>0$ and let $(x_n)_{n\ge 1}$ be a sequence of
positive real numbers such that $\lim_{n\to\infty}x_n=x$. Then
\[
W^{(\theta)}_n(x_n)\xrightarrow[n\to\infty]{} W^{(\theta)}(x).
\]
\end{lemma}
\begin{proof}
First note from Lemma 8.4 in \cite{K} that we can write
\begin{equation}
W^{(\theta)}_n(x) = {\rm e}^{\Phi_n(\theta)x}W_{\Phi_n(\theta)}(x), \qquad x\geq 0,
\label{Wcom}
\end{equation}
where $W_{\Phi_n(\theta)}$ is the scale function of the process $(X, \mathbb{P}^{\Phi_n(\theta),n})$.
As earlier noted, Lemma 8.2. of \cite{K} tells us that we may further write 
\[
W_{\Phi_n(\theta)}(x_n)=W_{\Phi_n(\theta)}(a)\exp\left\{-\int_{x_n}^{a}\eta_n(\bar{\epsilon}\geq
t)\ud t\right\}.
\]

Now appealing to Lemma \ref{eta-convergence} and (\ref{recall}), it is straightforward to show with the help of dominated convergence that 
\[
\begin{split}
\lim_{n\to\infty}W_{\Phi_n(\theta)}(x_n)&=\lim_{n\to\infty}W_{\Phi_n(\theta)}(a)\exp\left\{-\int_{x_n}^{a}\eta_n(\bar{\epsilon}\geq
t)\ud t\right\}\\
&=W_{\Phi(\theta)}(a)\exp\left\{-\int_{x}^{a}\eta(\bar{\epsilon}\geq
t)\ud t\right\}\\
&=W_{\Phi(\theta)}(x).
\end{split}
\]
Applying this limit to (\ref{Wcom}) completes the proof. 
\end{proof}

\begin{lemma}
Let $x>0$ and let $(x_n)_{n\ge 1}$ be a sequence of
positive real numbers such that $\lim_{n\to\infty}x_n=x$. Then for
any $\theta>0$
\[
W^{(\theta)\prime}_n(x_n)\xrightarrow[n\to\infty]{} W^{(\theta)\prime}(x).
\]
\end{lemma}
\begin{proof}
From (\ref{W'/W}) and the previous lemma, it suffices to show that 
\[
\eta^{\Phi_n(\theta)}_n(\overline{\epsilon}>x_n)\xrightarrow[n\to\infty]{} \eta^{\Phi(\theta)}(\overline{\epsilon}>x).
\]
However, this follows from Lemma \ref{eta-convergence}.
\end{proof}

Let us return to the proof of Theorem \ref{main} for the case of unbounded variation paths with $\sigma^2>0$ using the approximating sequence of bounded variation spectrally negative L\'evy processes constructed with L\'evy measure given by (\ref{approx-Pi}).  Starting from (\ref{fii2}), and taking account of the reasoning in the previous section, it suffices to show that the right hand side of (\ref{fii2}) converges to the ratio
\begin{equation}\label{ui}
%\mathbb{E}_b[{\rm e}^{-\theta\int_0^{\rho_{a}^+\wedge\rho_{c}^-}\mathbf{1}_{\{U_s<b\}}ds}]=
\frac{\displaystyle\frac{1}{\mathbb{W}(c-b)}+\displaystyle\frac{\sigma^2}{2}\mathcal{C}^{(\theta)}(a,b)+\int_{-\infty}^0\int_0^{\infty}{\mathcal A}^{(\theta)}(z,a,b,c,y)\Pi(dz-y){\rm d}y}{(\psi'(0+)-\delta)^++\displaystyle\frac{\sigma^2}{2}\mathcal{D}^{(\theta)}(a,b,c)+\int_{-\infty}^0\int_0^{\infty}\mathcal{B}^{(\theta)}(z,a,c,y)\Pi(dz-y){\rm d}y}.
\end{equation}

Proceeding exactly as in the proof of Theorem 6, it is easy to see
that
\begin{equation}
\lim_{n\to\infty}\int_{-\infty}^0\int_0^{\infty}{\mathcal A}^{(\theta),n}(z,a,b,c,y)\Pi_n(dz-y){\rm d}y=\int_{-\infty}^0\int_0^{\infty}{\mathcal A}^{(\theta)}(z,a,b,c,y)\Pi(dz-y){\rm d}y.\notag
\end{equation}
and
\begin{equation}
\lim_{n\to\infty}\int_{-\infty}^0\int_0^{\infty}\mathcal{B}^{(\theta),n}(z,a,b,c,y)\Pi_n(dz-y){\rm d}y=\int_{-\infty}^0\int_0^{\infty}{\mathcal A}^{(\theta)}(z,a,b,c,y)\Pi(dz-y){\rm d}y.\notag
\end{equation}

To obtain the terms $\sigma^2\mathcal{C}^{(\theta)}(a,b)/2$ in the numerator, consider the 
integral in the numerator of the right hand side of (\ref{ui}). We have
\begin{eqnarray*}
\lefteqn{n^2\sigma^2\int_0^{\infty}\int_{(-\infty,0)}{\mathcal A}^{(\theta),n}(z,a,b,c,y)\delta_{-1/n}({\rm d}z-y){\rm d}y}&&\\
&&=n^2\sigma^2\int_{0}^{\infty}\int_{(-\infty,-y)}{\mathcal A}^{\theta,n}(z+y,a,b,c,y)\delta_{-1/n}({\rm d}z){\rm d}y\notag\\
&&=n^2\sigma^2\int_{0}^{1/n}\frac{\mathbb{W}_n(c-b-y)}{\mathbb{W}_n(c-b)}\left(Z^{(\theta)}_n(y-1/n+b-a)-Z^{(\theta)}_n(b-a)\frac{W^{(\theta)}_n(y-1/n+b-a)}{W^{(\theta)}_n(b-a)}\right){\rm d}y.\notag
\end{eqnarray*}
It follows that 
\begin{eqnarray*}
\lefteqn{
n^2\sigma^2\int_0^{\infty}\int_{(-\infty,0)}{\mathcal A}^{(\theta),n}(z,a,b,c,y)\delta_{-1/n}({\rm d}z-y){\rm d}y}&& 
\\
&=&n^2\sigma^2\int_{0}^{1/n}\frac{\mathbb{W}_n(c-b-y)}{\mathbb{W}_n(c-b)}\left(Z^{(\theta)}_n(y-1/n+b-a)-Z^{(\theta)}_n(b-a)\frac{W^{(\theta)}_n(y-1/n+b-a)}{W^{(\theta)}_n(b-a)}\right){\rm d}y\\
%&=&n^2\sigma^2\int_{0}^{1/n}\frac{\mathbb{W}_n(c-b-y)}{\mathbb{W}_n(c-b)}\left(Z^{(\theta)}_n(y-1/n+b-a)-Z^{(\theta)}_n(b-a)\frac{W^{(\theta)}_n(y-1/n+b-a)}{W^{(\theta)}_n(b-a)}\right){\rm d}y\notag\\
&=&n^2\sigma^2\int_{0}^{1/n}y\frac{\mathbb{W}'_n(c-b-y)}{\mathbb{W}_n(c-b)}\left(Z^{(\theta)}_n(y-1/n+b-a)-Z^{(\theta)}_n(b-a)\frac{W^{(\theta)}_n(y-1/n+b-a)}{W^{(\theta)}_n(b-a)}\right){\rm d}y\notag\\
&&-n^2\sigma^2\int_{0}^{1/n}y\frac{\mathbb{W}_n(c-b-y)}{\mathbb{W}_n(c-b)}\left(\theta
W^{(\theta)}_n(y-1/n+b-a)-Z^{(\theta)}_n(b-a)\frac{W^{(\theta)\prime}_n(y-1/n+b-a)}{W^{(\theta)}_n(b-a)}\right){\rm d}y\notag\\
&=&\sigma^2\int_{0}^{1}u\frac{\mathbb{W}'_n(c-b-u/n)}{\mathbb{W}_n(c-b)}\left(Z^{(\theta)}_n((u-1)/n+b-a)-Z^{(\theta)}_n(b-a)\frac{W^{(\theta)}_n((u-1)/n+b-a)}{W^{(\theta)}_n(b-a)}\right){\rm d}u\notag\\
&&-\sigma^2\int_{0}^{1}y\frac{\mathbb{W}_n(c-b-u/n)}{\mathbb{W}_n(c-b)}\left(\theta
W^{(\theta)}_n((u-1)/n+b-a)-Z^{(\theta)}_n(b-a)\frac{W^{(\theta)\prime}_n((u-1)/n+b-a)}{W^{(\theta)}_n(b-a)}\right){\rm d}u\notag,
\end{eqnarray*}
where, in the second equality we have integrated by parts and in the third equality we have made the change of variable $u = ny$.

Appealing to the  Dominated Convergence Theorem
it now follows that 
\begin{eqnarray*}
\lefteqn{\hspace{-2cm}
\lim_{n\uparrow\infty} n^2\sigma^2\int_0^{\infty}\int_{(-\infty,0)}{\mathcal A}^{(\theta),n}(z,a,b,c,y)\delta_{-1/n}({\rm d}z-y){\rm d}y
} &&\\
 &=& \sigma^2\left(Z^{(\theta)}(b-a)\frac{W^{(\theta)\prime}(b-a)}{W^{(\theta)}(b-a)}- \theta W^{(\theta)}(b-a)\right)\int_{0}^{1}y{\rm d}y\\
&=&\frac{\sigma^2}{2}\left(Z^{(\theta)}(b-a)\frac{W^{(\theta)\prime}(b-a)}{W^{(\theta)}(b-a)}- \theta
W^{(\theta)}(b-a)\right)\\
&=&\frac{ \sigma^2}{2}\mathcal{C}^{(\theta)}(a,b).
\end{eqnarray*}

A similar computation, left to the reader, will reveal 
\begin{align}
\lim_{n\to\infty}n^2\sigma^2\int_0^{\infty}\int_{(-\infty,0)}&\mathcal{B}^{\theta,n}(z,a,b,c,y)\delta_{-1/n}(dz-y){\rm d}y=\frac{\sigma^2}{2}\left(\frac{\mathbb{W}'(c-b)}{\mathbb{W}(c-b)}+\frac{W^{(\theta)\prime}(b-a)}{W^{(\theta)}(b-a)}-\varphi(0)\right)\notag.
\end{align}
This concludes the proof of Theorem \ref{main}.\hfill$\square$

\section{Proofs of Corollaries \ref{corr1} and \ref{corr2}}

These corollaries are the result of taking limits as $c\uparrow\infty$ and $a\downarrow-\infty$ in the expression (\ref{ui}). In dealing with the left hand side of (\ref{ui}) one appeals to dominated convergence and the monotonicity of the stopping times $\rho^+_c$ and $\rho^-_a$ in their respective parameters. For the right hand side, one may make use of the  limits given below, together with dominated convergence. The  implementation of the Dominated Convergence Theorem is similar to the arguments used above, and we omit them for the sake of brevity, leaving the details to the reader.

We first note that for all $z\in (-\infty,0)$,
\[
\lim_{a\downarrow-\infty}Z^{(\theta)}(z+b-a)-Z^{(\theta)}(b-a)\frac{W^{(\theta)}(z+b-a)}{W^{(\theta)}(b-a)}=\lim_{a\downarrow-\infty}\mathbb{E}_{z+b}\left[{\rm e}^{-\theta\tau_a^-};\tau_a^-<\tau_{b}^+\right]=0,
\]
which implies
\[
\lim_{a\downarrow-\infty}{\mathcal A}^{(\theta)}(z,a,b,c,y)=0.
\]
On the other hand,  recall
\[
\lim_{a\downarrow-\infty}\frac{W^{(\theta)}(z+b-a)}{W^{(\theta)}(b-a)}={\rm e}^{\Phi(\theta)z},
\]
and from (\ref{W'/W})
\[
\lim_{a\downarrow-\infty}\frac{W^{(\theta)\prime}(b-a)}{W^{(\theta)}(b-a)} = \Phi(\theta).
\]
Hence 
\[
\lim_{a\downarrow-\infty}\mathcal{B}^{(\theta)}(z, a,b, c,y)={\rm e}^{-\varphi(0)y}-{\rm e}^{\Phi(\theta)z}\frac{\mathbb{W}(c-b-y)}{\mathbb{W}(c-b)}\mathbf{1}_{(a-b,0)}(z)\mathbf{1}_{(0,c-b)}(y).
\]

\[
\lim_{a\downarrow-\infty}\mathcal{D}^{(\theta)}(a,b,c)=\frac{\mathbb{W}'(c-b)}{\mathbb{W}(c-b)}+\Phi(\theta)-\varphi(0).
\]
Moreover, appealing to Theorem 2.8 (ii) in \cite{KKR}, we have 
\[
\lim_{a\downarrow-\infty}\mathcal{C}^{(\theta)}(a,b)  =  \lim_{a\downarrow-\infty}\mathbb{E}_0({\rm e}^{-\theta\sigma_{b-a}} ) = 0,
\]
where $\sigma_a = \inf\{t>0 : \overline{X}_t - X_t >a\}$.
From these limits we can easily deduce the statement in part (i) of Corollary \ref{corr1}.

Next recall, see for instance Exercise 8.5 in \cite{K}, that
\[
\lim_{z\uparrow\infty}\frac{\mathbb{W}(z-x)}{\mathbb{W}(z)}={\rm e}^{-\varphi(0)x}.
\]
From identity (8.7) in \cite{K}, we deduce
\[
\mathbb{W}(\infty)=\begin{cases}
\displaystyle\frac{1}{\psi^\prime(0+)-\delta}&\textrm{ if } \psi^\prime(0+)>\delta,\\
\infty &\textrm{ if }  \psi^\prime(0+)\le \delta.
\end{cases}
\]
Moreover, from the relation for $\mathbb{W}'(z)/\mathbb{W}(z)$ that is analogous to (\ref{W'/W}) we have
\[
\lim_{z\uparrow\infty}\frac{{\mathbb W}'(z)}{\mathbb{W}(z)} = \varphi(0).
\]
Hence, when $c$ goes to  $\infty$, we get
\[
\lim_{c\uparrow\infty}{\mathcal A}^{(\theta)}(z,a,b,c,y)=
{\rm e}^{-\varphi(0)y}\left(Z^{(\theta)}(z+b-a)-Z^{(\theta)}(b-a)\frac{W^{(\theta)}(z+b-a)}{W^{(\theta)}(b-a)}\right)\mathbf{1}_{(a-b,0)}(z)
%=\left(Z^{(\theta)}(z+b-a)-Z^{(\theta)}(b-a)\frac{W^{(\theta)}(z+b-a)}{W^{(\theta)}(b-a)}\right){\rm e}^{-\varphi(0)y}.
\]
\[
\lim_{c\uparrow\infty}\mathcal{B}^{(\theta)}(z, a,b, c,y)=
{\rm e}^{\varphi(0)y}\left(1-\frac{W^{(\theta)}(z+b-a)}{W^{(\theta)}(b-a)}\mathbf{1}_{(a-b,0)}(z)\right)
\]
%\[
%\lim_{c\uparrow\infty}\mathcal{C}^{(\theta)}(a,b)=\theta
%W^{(\theta)}(b-a)-Z^{(\theta)}(b-a)\frac{W^{(\theta)\prime}(b-a)}{W^{(\theta)}(b-a)},
%\]
and
\[
\lim_{c\uparrow\infty}\mathcal{D}^{(\theta)}(a,b,c)=\frac{W^{(\theta)\prime}(b-a)}{W^{(\theta)}(b-a)}.
\]
From these limits we can deduce the statement in part (ii) of Corollary \ref{corr1}.

\bigskip

Finally, for the proof of Corollary \ref{corr2},  one may proceed as above, taking limits in either of the two expressions given in Corollary \ref{corr1} (respectively as $c\uparrow\infty$ in part (i) or $a\downarrow-\infty$ in part (ii)), however one may also recover the identity by revisiting the proof of Theorem \ref{main}. Indeed, setting $c = +\infty$ and $a  = -\infty$, in which case one should understand $p = \mathbb{E}_b (\zeta<\infty)$, the proof goes through verbatim for the case that $X$ has bounded variation paths. Appealing to the approximation (\ref{uniformBertoin}), the identity (\ref{totaloccupation}) can easily be shown to be valid in the case that $X$ has unbounded variation case.

In order to obtain the density of  $\int_0^{\infty}\mathbf{1}_{\{U_s<b\}}\ud s$,   we first  recall that $\Phi$ is the Laplace exponent of the ascending ladder time of $X$,  which is a subordinator (see Theorem VII.1 in \cite{B}). Therefore $\Phi$ respects the relation
  \[
\frac{\Phi(\theta)}{\theta} =\texttt{a}+\int_0^\infty {\rm e}^{-\theta x}\overline{\Pi}_{\Phi}(x)\ud x,
 \]
 where $\texttt{a}\ge 0$, $\overline{\Pi}_{\Phi}(x)=\Pi_{\Phi}([x,\infty))$ and $\Pi_\Phi$ is the underlying L\'evy measure associated to $\Phi$.   Said another way, the quantity $\Phi(\theta)/\theta$, for $\theta\geq 0$ is the Laplace transform of the measure 
 \[
\mu({\rm d}x) =\texttt{a}\delta_{0}({\rm d}x) + \overline{\Pi}_{\Phi}(x){\rm d}x, \qquad x\geq 0,
 \]
 where $\delta_0({\rm d}x)$ is the Dirac-delta measure which places an atom at zero.
 
 Now rewriting  (\ref{totaloccupation}) and noting that  under the assumption $\psi^\prime(0+)>\delta$,  we necessarily have $\delta\Phi(\theta)<\theta$, we get
\[
\begin{split}
\mathbb{E}_b\bigg[\exp\bigg\{-\theta\int_0^{\infty}\mathbf{1}_{\{U_s<b\}}\ud s\bigg\}\bigg]
&= (\psi'(0+) - \delta)\frac{\Phi(\theta)}{\theta}\times\frac{1}{1 - \displaystyle\frac{\delta\Phi(\theta)}{\theta}}\\
&= \left(\frac{\psi'(0+)}{\delta} - 1\right)
\sum_{n\ge1}\delta^n\left(\frac{\Phi(\theta)}{\theta}\right)^n.
\end{split}
\]
%Since  the right-hand side of the above equation corresponds to the Laplace transform of the convolution of 
%\[
%f(x)=a\mathbf{1}_{\{x=0\}}+\overline{\Pi}_{\Phi}(x)\mathbf{1}_{\{x>0\}}\qquad \textrm{ and }\qquad \nu(z)=\sum_{n\ge 0}\mu^{(\ast n)}(z),
%\]
%where $\mu(z)=\delta\int_0^zf(y)\ud y$  and $\mu^{(\ast n)}(z)$ denotes the $n$-th convolution of $\mu$, 
We deduce that $\int_0^{\infty}\mathbf{1}_{\{U_s<b\}}\ud s$ has a  density which is given by
\[
\left(\frac{\psi'(0+)}{\delta} - 1\right)
\sum_{n\ge1}\delta^n\mu^{*n}({\rm d}x), \qquad x\geq 0,
\]
where $\mu^{*n}$ is the $n$-fold
convolution of $\mu$. Note that $\mu^{*n}(\{0\}) = \texttt{a}^n$ and hence, with $\nu({\rm d}x) = \overline\Pi_{\Phi}(x){\rm d}x$, we finally come to rest at 
\[
\mathbb{P}_b\left(\int_0^{\infty}\mathbf{1}_{\{U_s<b\}}\ud \in{\rm d}x\right)= 
\left(\frac{\psi'(0+)}{\delta} - 1\right)\left(\frac{\delta\texttt{a}}{1- \delta\texttt{a}}\delta_0({\rm d}x) + \mathbf{1}_{\{x>0\}}\sum_{n\geq 1}\delta^n\nu^{*n}({\rm d}x)\right),
\]
where $\nu^{*n}$ is the $n$-fold convolution of $\nu$.
\hfill$\square$

\section*{Acknowledgement}

J-L.P. acknowledges financial support from CONACyT grant number 129326.
J-C. P. acknowledges financial support from CONACyT grant number 128896
A.E.K acknowledges financial support from the Santander Research Fund.

\end{document}